\documentclass[11pt]{amsart}

\usepackage{latexsym}
\usepackage{amsmath}
\usepackage{amsfonts}
\usepackage{amssymb}
\usepackage{tikz,float}

\newtheorem{theorem}{Theorem}[section]
\newtheorem{lemma}[theorem]{Lemma}
\newtheorem{corollary}[theorem]{Corollary}
\newtheorem{proposition}[theorem]{Proposition}

\newtheorem{sublemma}{}[theorem]

\theoremstyle{definition}

\theoremstyle{remark}

\numberwithin{equation}{section}

\newcommand{\del}{\backslash}
\DeclareMathOperator{\cl}{cl}
\DeclareMathOperator{\si}{si}

\newsavebox\ideabox



\date{\today}

\begin{document}

\title[The Binary $\Theta_3$-closed Matroids]{A binary-matroid analogue of a graph connectivity theorem of Jamison and Mulder}

\author{Cameron Crenshaw}
\address{Mathematics Department\\
Louisiana State University\\
Baton Rouge, Louisiana}
\email{ccrens5@lsu.edu}

\author{James Oxley}
\address{Mathematics Department\\
Louisiana State University\\
Baton Rouge, Louisiana}
\email{oxley@math.lsu.edu}

\begin{abstract}
Jamison and Mulder characterized the set of graphs that can be built from cycles and complete graphs via 1-sums and parallel connections as those graphs $G$ such that, whenever two vertices $x$ and $y$ of $G$ are joined by three internally disjoint paths, $x$ and $y$ are adjacent. This paper proves an analogous result for the set of binary matroids constructible from direct sums and parallel connections of circuits, complete graphs, and projective geometries.
\end{abstract}

\begin{abstract}
Let $G$ be a graph such that, whenever two vertices $x$ and $y$ of $G$ are joined by three internally disjoint paths, $x$ and $y$ are adjacent. Jamison and Mulder determined that the set of such graphs coincides with the set of graphs that can be built from cycles and complete graphs via 1-sums and parallel connections. This paper proves an analogous result for binary matroids.
\end{abstract}

\maketitle

\section{Introduction}
\label{introduction}

Jamison and Mulder~\cite{JM} defined a graph $G$ to be \emph{$\Theta_3$-closed} if, whenever distinct vertices $x$ and $y$ of $G$ are joined by three internally disjoint paths, $x$ and $y$ are adjacent. For disjoint graphs $G_1$ and $G_2$, a 1-\emph{sum} of $G_1$ and $G_2$ is a graph that is obtained by identifying a vertex of $G_1$ with a vertex of $G_2$. Following Jamison and Mulder, we define a \emph{$2$-sum} of $G_1$ and $G_2$ to be a graph that is obtained by identifying an edge of $G_1$ with an edge of $G_2$. Note that, in contrast to some other definitions of this operation, we retain the identified edge as an edge of the resulting graph. The main result of Jamison and Mulder's paper is the following.

\begin{theorem}
\label{JMtheorem}
A connected graph $G$ is $\Theta_3$-closed if and only if $G$ can be built via $1$-sums and $2$-sums from cycles and complete graphs.
\end{theorem}

This paper generalizes Theorem~\ref{JMtheorem} to binary matroids; all matroids considered here are binary unless stated otherwise. The terminology and notation follow~\cite{oxley} with the following additions. We will use $P_r$ to denote the rank-$r$ binary projective geometry, $PG(r-1,2)$. A \emph{theta-graph} is a graph that consists of two distinct vertices and three internally disjoint paths between them. A theta-graph in a matroid $M$ is a restriction of $M$ that is isomorphic to the cycle matroid of a theta-graph. Equivalently, it is a restriction of $M$ that is isomorphic to a matroid that is obtained from $U_{1,3}$ by a sequence of series extensions. The series classes of a theta-graph are its \emph{arcs}. Let $T$ be a theta-graph of $M$ with arcs $A_1$, $A_2$, and $A_3$. If $M$ has an element $e$ such that, for every $i$, either $A_i\cup e$ is a circuit of $M$, or $A_i=\{e\}$, then $e$ \emph{completes} $T$ in $M$, and $T$ is said to be \emph{complete}. A matroid $M$ is \emph{matroid $\Theta_3$-closed} if every theta-graph of $M$ is complete. The next theorem is the main result of this paper.

\begin{theorem}
\label{mainresult}
A matroid $M$ is matroid $\Theta_3$-closed if and only if $M$ can be built via direct sums and parallel connections from circuits, cycle matroids of complete graphs, and projective geometries.
\end{theorem}

Suppose $M$ is isomorphic to the cycle matroid of a graph $G$. Two vertices in $G$ that are joined by three internally disjoint paths are adjacent via an edge $e$ exactly when the corresponding theta-graph of $M$ is completed by $e$. In other words, $G$ is $\Theta_3$-closed if and only if $M$ is matroid $\Theta_3$-closed. This allows us to refer to $M$ as \emph{$\Theta_3$-closed} without ambiguity. We will denote the class of $\Theta_3$-closed matroids by $\Theta_3$.

Section~\ref{prelims} introduces supporting results. The 3-connected matroids that are $\Theta_3$-closed are characterized in Section~\ref{3conn}, and the proof of Theorem~\ref{mainresult} appears in Section~\ref{main}.

\section{Preliminaries}
\label{prelims}

Our first proposition collects some essential properties of $\Theta_3$-closed matroids. These properties will be used frequently and often implicitly.

\begin{proposition}
\label{basics}
If $M\in \Theta_3$, then
\begin{itemize}
\item[(i)]{$\si(M)\in\Theta_3$;}
\item[(ii)]{$M\vert F\in\Theta_3$ for every flat $F$ of $M$; and}
\item[(iii)]{$M/e\in \Theta_3$ for every $e\in E(M)$.}
\end{itemize}
\end{proposition}

\begin{proof}
Parts (i) and (ii) are straightforward. For part (iii), let $T$ be a theta-graph of $M/e$. Then $[(M/e)\vert T]^\ast$ is obtained from $U_{2,3}$ by adding elements in parallel to the existing elements. Since $M$ is binary, it follows that $M^\ast/(E(M)-(T\cup e))$ is obtained from $[(M/e)\vert T]^\ast$, that is, from $M^\ast/(E(M)-(T\cup e))\del e$, by adding $e$ as a coloop or by adding $e$ in parallel to one of the existing elements. Thus, $e$ is either a loop in $M\vert (T\cup e)$, or is in series with another element. Hence, since $T$ is complete in $M$, it is complete in $M/e$.
\end{proof}

Evidently, a matroid is in $\Theta_3$ if and only if its connected components are in $\Theta_3$. This will also be used implicitly throughout the paper. The following is an immediate consequence of Proposition~\ref{basics}.

\begin{corollary}
\label{parallelminorclosed}
If $M\in\Theta_3$ and $N$ is a parallel minor of $M$, then $N\in\Theta_3$.
\end{corollary}

From, for example,~\cite[Exercise 8.3.3]{oxley}, if $M=M_1\oplus_2 M_2$, then $M_1$ and $M_2$ are parallel minors of $M$. The next result now follows from Corollary~\ref{parallelminorclosed}.

\begin{corollary}
\label{2summands}
If $M\oplus_2 N$ is in $\Theta_3$, then $M$ and $N$ are in $\Theta_3$.
\end{corollary}

To see that the converse of the last corollary fails, observe that $M(K_{2,4})$ is not in $\Theta_3$ although it is the 2-sum of two copies of a matroid in $\Theta_3$.

We conclude this section with a result about constructing larger matroids in $\Theta_3$ from smaller ones. Recall that, for sets $X$ and $Y$ in a matroid $M$, the \emph{local connectivity} between $X$ and $Y$, denoted $\sqcap(X,Y)$, is defined by $\sqcap(X,Y)=r(X)+r(Y)-r(X\cup Y)$. We will use the following result about local connectivity from, for example,~\cite[Lemma 8.2.3]{oxley}.

\begin{lemma}
\label{staplelemma}
Let $X_1$, $X_2$, $Y_1$, and $Y_2$ be subsets of the ground set of a matroid $M$. If $X_1\supseteq Y_1$ and $X_2\supseteq Y_2$, then $\sqcap(X_1,X_2)\geq\sqcap(Y_1,Y_2)$.
\end{lemma}

\begin{proposition}
\label{parconnprop}
For matroids $M$ and $N$, the parallel connection $P(M,N)$ is in $\Theta_3$ if and only if $M\in \Theta_3$ and $N\in \Theta_3$.
\end{proposition}

\begin{proof}
Let $p$ be the basepoint of the parallel connection. When $p$ is a loop or a coloop of $M$, the matroid $P(M,N)$ is $M\oplus (N/p)$ or $(M/ p)\oplus N$, respectively. In these cases, it follows using Proposition~\ref{basics} that the result holds. Thus we may assume that $p$ is neither a loop nor a coloop of $M$ or $N$. Suppose $P(M,N)\in\Theta_3$. Let $B_M$ be a basis for $M$ containing $p$. Extend $B_M$ to a basis $B$ for $P(M,N)$. After contracting both the elements of $B-B_M$ in $P(M,N)$ as well as all of the resulting loops, the remaining elements of $E(N)-p$ are parallel to $p$. We deduce that $M$, and similarly $N$, is a parallel minor of $P(M,N)$. Hence, by Corollary~\ref{parallelminorclosed}, $M$ and $N$ are in $\Theta_3$.

Conversely, suppose that $M, N\in \Theta_3$ and let $T$ be a theta-graph of $P(M,N)$ with arcs $A_1$, $A_2$, and $A_3$. Then we may assume that $\vert A_i\vert\geq 2$ for each $i$, otherwise $T$ is complete. Suppose $p\in A_1$. Then $A_1\cup A_2$ is a circuit containing $p$, so it is contained in $E(M)$ or $E(N)$ depending on which of these sets contains $A_1-p$. It follows that the same set contains $A_2$ and, likewise $A_3$, so $T$ is complete. Hence we may assume that $p\notin T$.

Suppose that each of the arcs of $T$ meets both $E(M\del p)$ and $E(N\del p)$. Let $T_M=E(T)\cap E(M)$, and similarly for $T_N$. Note that $T_M$ and $T_N$ are independent, and $T_M\cup T_N=E(T)$, so
\begin{align*}
\sqcap(T_M,T_N)&=r(T_M)+r(T_N)-r(T_M\cup T_N)\\
&=\vert T_M\vert + \vert T_N\vert - (\vert T_M\vert + \vert T_N\vert - 2)\\
&=2.
\end{align*}
However, $\sqcap(E(M),E(N))=1$, contradicting Lemma~\ref{staplelemma}.

Next, suppose that each of $A_1$ and $A_2$ meets both $E(M\del p)$ and $E(N\del p)$. Then, from above, we may assume that $A_3\subseteq E(M\del p)$. The circuits $A_1\cup A_3$ and $A_2\cup A_3$ have the form $(C_1-p)\cup(D_1-p)$ and $(C_2-p)\cup(D_2-p)$, respectively, for circuits $C_1$ and $C_2$ of $M$ containing $p$, and circuits $D_1$ and $D_2$ of $N$ containing $p$. Because $A_3\subseteq E(M)$ and $A_1\cap A_2=\emptyset$, it follows that $D_1-p$ and $D_2-p$ are disjoint. However, since $M$ is binary, $D_1\triangle D_2$ contains a circuit of $P(M,N)$ that is properly contained in the circuit $A_1\cup A_2$, a contradiction.

Now, suppose that $A_1$ meets both $E(M\del p)$ and $E(N\del p)$. Then, from above, each of the remaining arcs of $T$ lies in $E(M\del p)$ or $E(N\del p)$. We may assume that $A_2\subseteq E(M\del p)$. Suppose $A_3\subseteq E(N\del p)$. Then the circuits $A_1\cup A_2$ and $A_3\cup A_2$ have the form $(C_1-p)\cup (D_1-p)$ and $(C_3-p)\cup(D_3-p)$, respectively, for circuits $C_1$ and $C_3$ of $M$ containing $p$, and circuits $D_1$ and $D_3$ of $N$ containing $p$. Now, since $A_2\subseteq E(M\del p)$ and $A_1$ meets $E(M\del p)$, the set $C_1-p$ properly contains $A_2$. Further, as $A_3$ does not meet $E(M\del p)$, we have that $A_2=C_3-p$. This means $A_2\cup p$ is the circuit $C_3$, but $A_2\cup p$ is properly contained in $C_1$, a contradiction.

We conclude that $A_3\subseteq E(M\del p)$. Form $T'$ from $T$ by replacing the portion of $A_1$ in $E(N\del p)$ by $p$. Observe that $T'$ is isomorphic to a series minor of $T$, so $T'$ is a theta-graph. Moreover, $T'$ is a theta-graph of $M$, so it is completed in $M$ by an element $f$. Now, since $T$ and $T'$ share an arc, $f$ also completes $T$ in $P(M,N)$.

We are left to consider the case when each arc of $T$ is contained in either $E(M)$ or $E(N)$. If all three arcs belong to $E(M)$, say, then $T$ is complete in $M$, and so is complete in $P(M,N)$. Otherwise, $p$ completes $T$.
\end{proof}

\section{The $3$-Connected $\Theta_3$-closed Matroids}
\label{3conn}

The proof of Theorem~\ref{mainresult} will use the canonical tree decomposition of Cunningham and Edmonds~\cite{cunned} and, in support of that approach, this section proves the following 3-connected form of Theorem~\ref{mainresult}.

\begin{theorem}
\label{mainres3conn}
Let $M$ be a simple $3$-connected $\Theta_3$-closed matroid. Then $M$ is a projective geometry or the cycle matroid of a complete graph.
\end{theorem}

The proof of this theorem relies on the next two propositions.

\begin{proposition}
\label{mkorpgprop}
If $M$ is a simple matroid in $\Theta_3$ and $M$ has a spanning $M(K_{r+1})$-restriction, then $M\cong M(K_{r+1})$ or $M\cong P_r$.
\end{proposition}

\begin{proof}
Take a standard binary representation for $P_r$, and view $M$ as the restriction of $P_r$ to the set $X$ of vectors. Recall that the number of nonzero entries of a vector is its \emph{weight}, and that the \emph{distance} between two vectors is the number of coordinates upon which they disagree. Because $M$ has an $M(K_{r+1})$-restriction, we may assume that $X$ contains the set $Z$ of vectors of weight one or two. We may assume that $Z\neq X$. Then $M$ has an element $e$ of weight at least three. We shall establish that $M\cong P_r$ by proving the following three assertions.
\begin{itemize}
\item[(i)]{$M$ has an element of weight three;}
\item[(ii)]{if the matroid $M$ has every element of weight $k-1$ and an element of weight $k$, for some $k$ exceeding two, then $M$ has every element of weight $k$; and}
\item[(iii)]{if $M$ has every element of weight $k$, where $3\leq k< r$, then $M$ has an element of weight $k+1$.}
\end{itemize}

Let $e_i$ denote the weight-1 element whose nonzero entry is in the $i$th position. To show (i), we may assume $e$ has weight $k\geq 4$. Say $e=e_1+e_2+\cdots+e_k$. Let $Y=\{e,e_1,e_2,e_4,e_5,\dots,e_k,e_1+e_3,e_2+e_3\}$. Then $M\vert Y$ is a theta-graph having arcs $\{e,e_4,e_5,\dots,e_k\}$, $\{e_1,e_2+e_3\}$, and $\{e_2,e_1+e_3\}$. This theta-graph forces $e_1+e_2+e_3$ to be an element of $M$, so (i) holds.

To prove (ii), we may assume $k<r$. Suppose $g$ is an element of weight $k$ not in $M$, and let $f$ be an element of weight $k$ in $M$ with minimum distance from $g$. Let $s$ label a row where $f$ is 1 and $g$ is 0, and let $t$ label a row where $g$ is 1 and $f$ is 0. Next, as $k\geq 3$, there are two additional rows, $u$ and $v$, distinct from $s$ where $f$ is 1. Now, the set $\{f,e_u,e_v,e_s+e_t\}$ is independent, so the arcs $\{f, e_s+e_t\}$, $\{e_u, f+e_u+e_s+e_t\}$, and $\{e_v,f+e_v+e_s+e_t\}$ form a theta-graph in $M$. This theta-graph implies that $f+e_s+e_t$ belongs to $M$. However, $f+e_s+e_t$ has weight $k$ and is a smaller distance from $g$ than $f$, a contradiction. Thus (ii) holds.

Finally, let $f$ be an element of weight $k+1$ for some $k$ with $3\leq k<r$. By symmetry, we may assume that the set of rows in which $f$ is nonzero contains $\{1,2,3\}$. Then $\{f,e_1,e_2,e_3\}$ is independent in $M$, and the sets $\{e_1,f+e_1\}$, $\{e_2,f+e_2\}$, and $\{e_3,f+e_3\}$ are the arcs of a theta-graph in $M$. This theta-graph shows that $f$ belongs to $M$. Thus (iii) holds. Hence the proposition holds as well.
\end{proof}

The second proposition that we use to prove Theorem~\ref{mainres3conn} will follow from the following three results.

\begin{lemma}
\label{pglift}
Let $M$ be a simple rank-$r$ matroid in $\Theta_3$. Suppose that
\begin{itemize}
\item[(i)]$r\geq 4$;
\item[(ii)] $E(M)$ has a subset $P$ such that $M\vert P\cong P_{r-1}$; and
\item[(iii)]$E(M)-P$ contains at least three elements.
\end{itemize}
Then $M\cong P_r$.
\end{lemma}

\begin{proof}
View $M$ as a restriction of $P_r$, and let $\{e,f,g\}$ be a subset of $E(M)-P$. Let $p$ be a point in $E(P_r)- P$ that is not in $\{e,f,g\}$. Observe that, for each $x$ in $\{e,f,g\}$, the third point on the line in $P_r$ containing $\{x,p\}$ is in $P$. Thus there are three lines of $M$ that meet at $p$. Provided $p$ is not coplanar with $\{e,f,g\}$, these lines define a theta-graph in $M$ that is completed by $p$, so $p$ is in $M$. It remains to show that the point $q$ of $E(P_r)- (P\cup \{e,f,g\})$ that is coplanar with $\{e,f,g\}$ belongs to $M$. But one easily checks that $P_r\del q$ is not in $\Theta_3$ when $r\geq 4$. Thus $M\cong P_r$.
\end{proof}

\begin{corollary}
\label{prfromf7s}
Let $M$ be a simple rank-$r$ matroid in $\Theta_3$ with $r\geq 3$. If $M$ has a basis $B$ and an element $b$ in $B$ so that, for each $\{x,y\}\subseteq B-b$, the set $\{b,x,y\}$ spans an $F_7$-restriction of $M$, then $M\cong P_r$.
\end{corollary}

\begin{proof}
Let $B=\{b_1,b_2,\dots,b_r\}$ with $b=b_1$. If $r=3$, then the result is immediate, so suppose $r\geq 4$. By induction, $M\vert\cl(B-b_r)$ is isomorphic to $P_{r-1}$. Since $M\vert\cl(\{b_1,b_2,b_r\})\cong F_7$, we see that this restriction contains an independent set of three elements that avoids $\cl(B-b_r)$. Lemma~\ref{pglift} now implies that $M\cong P_r$.
\end{proof}

The next result was proved by McNulty and Wu~\cite[Lemma 2.10]{mcnultywu}.

\begin{lemma}
\label{connhyp}
Let $M$ be a $3$-connected binary matroid with at least four elements. Then, for any two distinct elements $e$ and $f$ of $M$, there is a connected hyperplane containing $e$ and avoiding $f$.
\end{lemma}

For a simple binary matroid $M$, we now define the smallest $\Theta_3$-closed matroid whose ground set contains $E(M)$. Let $M_0=M$ and $r(M)=r$. Suppose $M_0, M_1, \dots, M_k$ have been defined. The simple binary matroid $M_{k+1}$ is obtained from $M_k$ by ensuring that, whenever $T$ is an incomplete theta-graph of $M_k$, the element $x$ that completes $T$ is in $E(M_{k+1})$. Since each $M_i$ is a restriction of $P_r$, there is a $j$ for which $M_{j+1}=M_j$. When this first occurs, we call $M_j$ the \emph{$\Theta_3$-closure} of $M$. Evidently this is well defined. By associating $M$ with its ground set, the $\Theta_3$-closure is a closure operator (but not necessarily a matroid closure operator) on the set of subsets of the ground set of any projective geometry containing $M$.

\begin{proposition}
\label{pginflater}
Let $M$ be a simple $3$-connected matroid in $\Theta_3$, and let $k$ be an integer exceeding two. If $M$ has a simple minor $N$ whose $\Theta_3$-closure is $P_k$, then $M$ is a projective geometry.
\end{proposition}

\begin{proof}
Take subsets $X$ and $Y$ of $E(M)$ such that $M/X\del Y=N$ with $X$ independent and $Y$ coindependent in $M$. The matroid $M/X$ is in $\Theta_3$ and has $N$ as a spanning restriction. Therefore $M/X$ has $P_k$ as a restriction, so $P_k$ is a minor of $M$. From here, the proof is by induction on the rank, $r$, of $M$.

If $r=k$, the result is immediate, so assume $r>k$. By Seymour's Splitter Theorem, $P_k$ can be obtained from $M$ by a sequence of single-element contractions and deletions, all while staying 3-connected. Let $e$ be the first element that is contracted in this sequence. Note that $\si(M/e)$ is a 3-connected member of $\Theta_3$ that has $P_k$ as a minor. By induction, $\si(M/e)\cong P_{r-1}$. Fix an embedding of $M$ in $P_{r}$. We may assume that $M\not\cong P_r$. Then some line $\ell$ of $P_{r}$ through $e$ is not contained in $E(M)$. For each subset $Z$ of $E(P_r)$, let $\cl_P(Z)$ be its closure in $P_r$. Since $\si(M/e)\cong P_{r-1}$, there is an element $s$ of $E(M)$ that is in $\ell-\{e\}$. Let $t$ be the point of $P_{r}$ in $\ell-\{e\}$ that is not in $E(M)$.

\begin{sublemma}
\label{ellpyramids}
Let $F$ be a rank-$4$ flat of $M$ containing $\ell-t$. Then $M\vert F$ is isomorphic to one of $P(F_7,U_{2,3})$ or $F_7\oplus U_{1,1}$ where the $F_7$-restriction of $M\vert F$ contains $s$ but avoids $e$.
\end{sublemma}

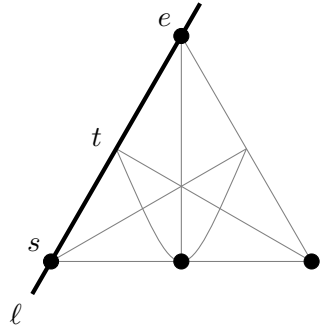
\begin{figure}
\begin{tikzpicture}

\foreach \x in {0,1,2} \draw[gray] (90+120*\x:2)--(-90+120*\x:1) (90+120*\x:2)--(210+120*\x:2);
\draw[gray] (150:1).. controls (-90:1.5) ..(30:1);

\foreach \x in {0,1,2} \draw[fill=black] (90+120*\x:2) circle (0.1);
\foreach \x in {0} \draw[fill=black] (-90+120*\x:1) circle (0.1);

\draw[ultra thick, shorten <= -0.5cm, shorten >= -0.5cm] (90:2)--(210:2);

\draw (90:2) node [above left] {$e$};
\draw (210:2) node[above left] {$s$};
\draw (150:1.3) node {$t$};

\draw (-2.2,-1.7) node {$\ell$};

\end{tikzpicture}
\caption{In the proof of \ref{ellpyramids}, each of $M\vert \pi_2$ and $M\vert \pi_3$ has this form.}
\label{ellplanes}
\end{figure}

To see this, first note that, by Proposition~\ref{basics}(ii), $M\vert F$ is in $\Theta_3$. Recall that each line of $\cl_P(F)$ through $e$ contains another point of $F$. Then there are three planes, $\pi_1$, $\pi_2$, and $\pi_3$, of $M\vert F$ containing $\ell-t$. Let $\mathcal{P}$ be this set of planes. Therefore, each plane in $\mathcal{P}$ has at least one pair of points that are not on $\ell$ such that these two points are collinear with either $s$ or $t$. Call such a pair of points a \emph{target pair}. The rest of the proof frequently relies of finding theta-graphs, particularly in rank 4. It maybe helpful to note that such a theta-graph will be isomorphic to $M(K_{2,3})$. Geometrically, this is three non-coplanar lines, each full except for shared point. This common point completes the theta-graph.

Suppose $\pi_1$ has a target pair collinear with $t$. Note that if two distinct planes in $\mathcal{P}$ each have a target pair collinear with $t$, then these pairs, along with $\{e,s\}$, are the arcs of an incomplete theta-graph in $M\vert F$, an impossibility. Consequently, neither $\pi_2$ nor $\pi_3$ has a target pair collinear with $t$, so both have the form in Figure~\ref{ellplanes}. The restriction of $M$ to $\pi_2\cup \pi_3$ is given in Figure~\ref{pi2pi3}.

The plane $\pi_1$ adds a pair of points collinear with $t$ to $M\vert (\pi_2\cup\pi_3)$, and one readily checks that the addition of this pair gives a restriction of $M\vert F$ that is isomorphic to $F_7\oplus_2 U_{2,3}$. A theta-graph of $M\vert F$ now gives that $M\vert F$ has a restriction isomorphic to $P(F_7,U_{2,3})$. It follows that $M\vert F\cong P(F_7, U_{2,3})$ otherwise, by Lemma~\ref{pglift}, $M\vert F\cong P_4$ and we obtain the contradiction that $t\in F$.

We may now suppose that no $\pi_i$ has a target pair collinear with $t$. It follows that each target pair in each $\pi_i$ is collinear with $s$, and that $e$ is the only element of $F$ outside of the target pairs. Observe that the target pairs must be coplanar as, otherwise, we can find an incomplete theta-graph in $M\vert F$. Thus $M\vert F\cong F_7\oplus U_{1,1}$. Noting that $e$ is not in the $F_7$-restriction of $M\vert F$, we conclude that~\ref{ellpyramids} holds.

\begin{figure}
\centering
\begin{tikzpicture}[style=thick]

\draw[fill=black] (-3.5,2.5) circle (0.1);
\draw[fill=black] (3.5,2.5) circle (0.1);
\draw[fill=black] (-1.75,2.25) circle (0.1);
\draw[fill=black] (1.75,2.25) circle (0.1);
\draw[fill=black] (0,2) circle (0.1);
\draw (0,2) node [below right] {$s$};
\draw[fill=black] (0,5) circle (0.1);
\draw (0,5) node [above right] {$e$};
\draw (0,3.5) node [above right] {$t$};

\draw (-5,4.5)--(0,5.5)--(5,4.5)--(5,0.5)--(0,1.5)--(-5,0.5)--(-5,4.5);
\draw (0,1.5)--(0,5.5);
\draw (-3.5,2.5)--(0,5) (-3.5,2.5)--(0,2);
\draw (3.5,2.5)--(0,5) (3.5,2.5)--(0,2);
\draw (-3.5,2.5)--(0,3.5)--(3.5,2.5);
\draw (1.75,3.75)--(0,2)--(-1.75,3.75);
\draw (-1.75,2.25)--(0,5)--(1.75,2.25);
\draw (-1.75,3.75).. controls (-2,1.75) ..(0,3.5);
\draw (1.75,3.75).. controls (2,1.75) ..(0,3.5);

\end{tikzpicture}
\caption{The matroid $M\vert(\pi_2\cup\pi_3)$ in the proof of~\ref{ellpyramids}.}
\label{pi2pi3}
\end{figure}
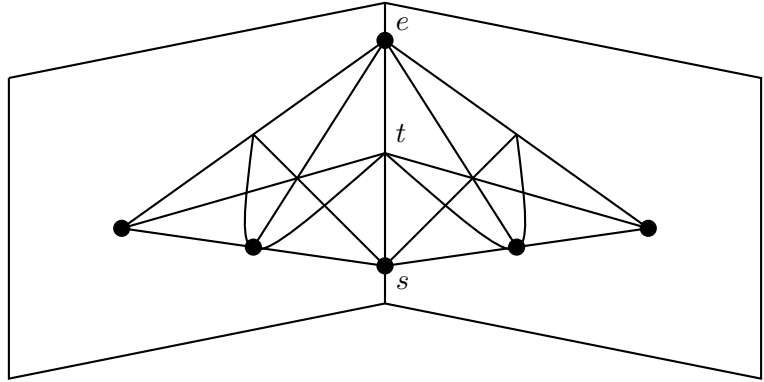

Since $M$ is 3-connected, it follows from~\ref{ellpyramids} that $r\geq 5$. By Lemma~\ref{connhyp}, there is a connected hyperplane $H$ of $M$ containing $s$ and avoiding $e$. Let $s=b_1$ and let $\{b_1,b_2,\dots,b_{r-1}\}$ be a basis $B$ of $M\vert H$. For distinct elements $i$ and $j$ of $\{2,3,\dots,r-1\}$, let $F_{i,j}$ denote the rank-4 flat $M\vert\cl(\{e,s,b_i,b_j\})$ of $M$. By~\ref{ellpyramids}, $M\vert  F_{i,j}$ is isomorphic to $F_7\oplus U_{1,1}$ or $P(F_7,U_{2,3})$. Let $X$ be the subset of $F_{i,j}$ such that $M\vert X\cong F_7$, and recall that $e\not\in X$ and $s\in X$. The hyperplane $H$ either contains $X$ or meets $X$ along one of the lines $\cl(\{s,b_i\})$ or $\cl(\{s,b_j\})$. We deduce the following.

\begin{sublemma}
\label{sswitcher}
For each pair $\{i,j\} \subseteq \{2,3,\dots,r-1\}$, at least one of $s+b_i$ or $s+b_j$ is in $E(M)$.
\end{sublemma}

Suppose $s+b_2$ is not in $E(M)$. By~\ref{sswitcher}, the element $s+b_i$ belongs to $E(M)$ for every $i$ in $\{3,4,\dots,r-1\}$. Consequently, for each pair $\{i,j\}$ in $\{3,4,\dots,r-1\}$, the hyperplane $H$ contains the copy of $F_7$ in $M\vert F_{i,j}$, and this $F_7$ is spanned by $\{s,b_i, b_j\}$. Corollary~\ref{prfromf7s} now implies that $M\vert \cl(B-b_2)\cong P_{r-2}$.

Now, since $M\vert H$ is connected, $H$ contains an element $f$ that is not in $\cl(B-b_2)\cup b_2$. The line $\cl(\{b_2,f\})$ meets the projective geometry $\cl(B-b_2)$, so $b_2+f$ is also in $M$. Now consider $Y=\cl(\{e,s,b_2,b_2+f\})$. The intersection $H\cap Y$ contains the line $\{b_2,f,b_2+f\}$ and also the element $s$. By applying~\ref{ellpyramids} to $M\vert Y$, we see that $H\cap Y$ is an $F_7$-restriction containing $s$ and $b_2$. Since $s+b_2$ is not in $E(M)$, we have a contradiction.

We conclude that $s+b_i$ is in $E(M)$ for every $i$ in $\{2,3,\dots,r-1\}$. The flat $M\vert F_{i,j}$ now meets $H$ in an $F_7$-restriction for every pair $\{i,j\}\subseteq \{2,3,\dots,r-1\}$ so, by Corollary~\ref{prfromf7s}, $M\vert H$ is a projective geometry. Finally, as $M$ is 3-connected, there is an independent set of three elements in $E(M)$ avoiding $H$. Hence, by Lemma~\ref{pglift}, $M\cong P_r$.
\end{proof}

We are now ready to prove the main result of this section.

\begin{proof}[Proof of Theorem~\ref{mainres3conn}]
Let $r$ be the rank of $M$. If $M$ is graphic, then Theorem~\ref{JMtheorem} gives that $M\cong M(K_{r+1})$, so we may assume that $M$ is not graphic. Thus $M$ has a minor $N$ isomorphic to $F_7$, $F_7^\ast$, $M^\ast(K_{3,3})$, or $M^\ast(K_5)$. By Proposition~\ref{pginflater}, it now suffices to show that the $\Theta_3$-closure, $\Theta(N)$, of $N$ is a projective geometry.

This is immediate when $N\cong F_7$, so suppose $N$ is isomorphic to $F_7^\ast$, labelled as in Figure~\ref{f7star}. The theta-graphs of $N$ imply that, in $\Theta(N)$, the plane containing $\{1,2,5,6\}$ is isomorphic to $F_7$. Proposition~\ref{pginflater} now implies that $M$ is isomorphic to $P_r$.

\begin{figure}
\centering
\begin{tikzpicture}[style=thick]

\draw[fill=black] (-3.5,2.5) circle (0.1);
\draw (-3.55,2.5) node[left] {1};
\draw[fill=black] (-1.75,3.75) circle (0.1);
\draw (-1.75,3.75) node[above left] {2};
\draw[fill=black] (1.75,3.75) circle (0.1);
\draw (1.75,3.75) node[above right] {3};
\draw[fill=black] (3.5,2.5) circle (0.1);
\draw (3.55,2.5) node[right] {4};
\draw[fill=black] (-1.75,2.25) circle (0.1);
\draw (-1.75,2.25) node [below left] {5};
\draw[fill=black] (-1.15,3.16) circle (0.1);
\draw (-1.2,3.55) node {6};
\draw[fill=black] (1.15,3.16) circle (0.1);
\draw (1.25,2.8) node {7};

\draw (-5,4.5)--(0,5.5)--(5,4.5)--(5,0.5)--(0,1.5)--(-5,0.5)--(-5,4.5);
\draw (0,1.5)--(0,5.5);
\draw (-3.5,2.5)--(0,5) (-3.5,2.5)--(0,2);
\draw (3.5,2.5)--(0,5);
\draw (-3.5,2.5)--(0,3.5)--(3.5,2.5);
\draw (1.75,3.75)--(0,2)--(-1.75,3.75);
\draw (-1.75,2.25)--(0,5);
\draw (-1.75,3.75).. controls (-2,1.75) ..(0,3.5);

\end{tikzpicture}
\caption{The matroid $F_7^\ast$.}
\label{f7star}
\end{figure}
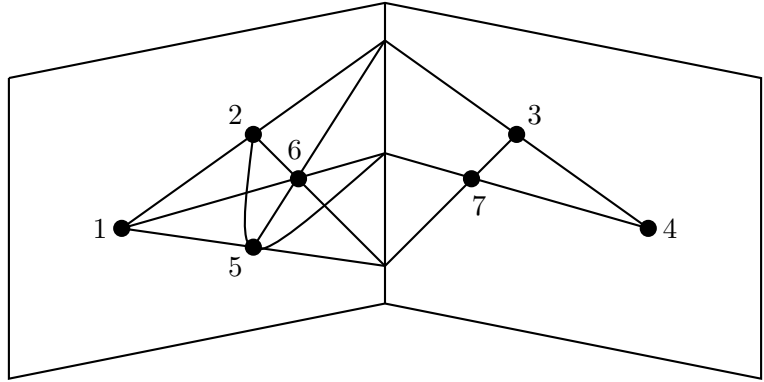

Next, suppose $N\cong M^\ast(K_{3,3})$. The complement of $N$ in $P_4$ is $U_{2,3}\oplus U_{2,3}$; let $x$ be an element of this complement. The elementary quotient of $N$ obtained by extending $N$ by $x$ and then contracting $x$ is shown in Figure~\ref{mk33quo}. The three pairwise-skew 2-element circuits of this quotient correspond to three lines in the extension of $N$ by $x$ where the union of these lines has rank four. Thus $N$ contains a theta-graph that is completed by $x$. It follows that $x$ and, symmetrically, every point in the complement of $N$ in $P_4$, belongs to $\Theta(N)$. Lemma~\ref{pglift} now implies that $\Theta(N)\cong P_4$.

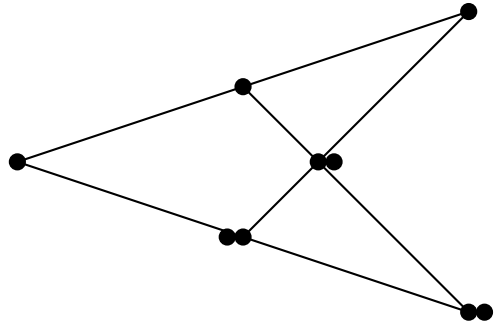
\begin{figure}
\begin{tikzpicture}[style=thick]

\draw (-3,0)--(3,2);
\draw (-3,0)--(3,-2);
\draw (0,1)--(3,-2);
\draw (0,-1)--(3,2);

\draw[fill=black] (-3,0) circle (0.1);
\draw[fill=black] (0,1) circle (0.1);
\draw[fill=black] (0,-1) circle (0.1);
\draw[fill=black] (3,2) circle (0.1);
\draw[fill=black] (3,-2) circle (0.1);
\draw[fill=black] (1,0) circle (0.1);

\draw[fill=black] (1.21,0) circle (0.1);
\draw[fill=black] (-0.21,-1) circle (0.1);
\draw[fill=black] (3.21,-2) circle (0.1);

\end{tikzpicture}
\caption{A quotient of $M^\ast(K_{3,3})$.}
\label{mk33quo}
\end{figure}

\begin{figure}
\centering
\begin{tikzpicture}[style=thick]

\foreach \x in {0,1,2,3,4} \draw (90+72*\x:2)--(162+72*\x:2);
\draw (90:2)--(234:2)--(18:2)--(162:2)--(306:2)--(90:2);

\foreach \x in {0,1,2,3,4} \draw[fill=white] (90 + 72*\x:2) circle (0.1);

\draw (-90:1.9) node {6};
\draw (198:1.9) node {3};
\draw (185:1.3) node {2};
\draw (139:1.3) node {1};
\draw (126:1.9) node {7};
\draw (113:1.3) node {8};
\draw (67:1.3) node {9};
\draw (54:1.9) node {0};
\draw (-5:1.3) node {4};
\draw (-18:1.9) node {5};

\end{tikzpicture}
\caption{The graph $K_5$.}
\label{k5}
\end{figure}
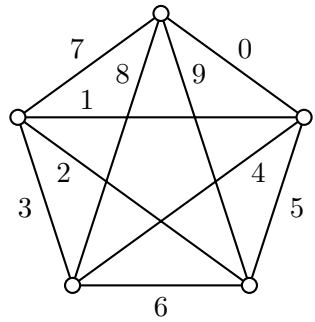
\begin{figure}
\begin{align*}
\begin{tabular}{rrrrrrrrrr}
1 & 2 & 3 & 4 & 5 & 6 & 7 & 8 & 9 & 0
\end{tabular}\hspace{0.25cm}\\
\left[\begin{array}{rrrrrr|rrrr}
1 & 0 & 0 & 0 & 0 & 0 & 1 & 0 & 0 & 1\\
0 & 1 & 0 & 0 & 0 & 0 & 1 & 0 & 1 & 0\\
0 & 0 & 1 & 0 & 0 & 0 & 1 & 1 & 0 & 0\\
0 & 0 & 0 & 1 & 0 & 0 & 0 & 1 & 0 & 1\\
0 & 0 & 0 & 0 & 1 & 0 & 0 & 0 & 1 & 1\\
0 & 0 & 0 & 0 & 0 & 1 & 0 & 1 & 1 & 0
\end{array}\right]
\end{align*}
\caption{A binary representation of $M^\ast(K_5)$.}
\label{m*k5}
\end{figure}

Now suppose $N\cong M^\ast(K_5)$, where $K_5$ is labelled as in Figure~\ref{k5}. Figure~\ref{m*k5} gives a corresponding binary representation of $N$. The dual of a theta-graph with arcs of size at least two is a triangle with no trivial parallel classes. Therefore, we can detect theta-graph restrictions of $N$ by contracting elements of $N^\ast$ to produce such a triangle. For example, $N^\ast/5,8$ is the dual of a theta-graph in $N$ with arcs $\{1,2\}$, $\{3,7\}$, and $\{0,4,6,9\}$. This theta-graph is completed by the element $[1\ 1\ 0\ 0\ 0\ 0]^T$, so this element belongs to $\Theta(N)$. The following is a list of duals of theta-graphs of $N$ and the corresponding elements of $\Theta(N)$ that they produce using this reasoning.

\begin{itemize}
\item $N^\ast/5,8$ gives $[1\ 1\ 0\ 0\ 0\ 0]^T\in \Theta(N)$.
\item $N^\ast/4,9$ gives $[1\ 0\ 1\ 0\ 0\ 0]^T\in \Theta(N)$.
\item $N^\ast/3,9$ gives $[1\ 0\ 0\ 1\ 0\ 0]^T\in \Theta(N)$.
\item $N^\ast/2,8$ gives $[1\ 0\ 0\ 0\ 1\ 0]^T\in \Theta(N)$.
\item $N^\ast/0,6$ gives $[0\ 1\ 1\ 0\ 0\ 0]^T\in \Theta(N)$.
\item $N^\ast/1,8$ gives $[0\ 1\ 0\ 0\ 1\ 0]^T\in \Theta(N)$.
\item $N^\ast/0,3$ gives $[0\ 1\ 0\ 0\ 0\ 1]^T\in \Theta(N)$.
\item $N^\ast/1,9$ gives $[0\ 0\ 1\ 1\ 0\ 0]^T\in \Theta(N)$.
\item $N^\ast/0,2$ gives $[0\ 0\ 1\ 0\ 0\ 1]^T\in \Theta(N)$.
\item $N^\ast/6,7$ gives $[0\ 0\ 0\ 1\ 1\ 0]^T\in \Theta(N)$.
\item $N^\ast/5,7$ gives $[0\ 0\ 0\ 1\ 0\ 1]^T\in \Theta(N)$.
\item $N^\ast/4,7$ gives $[0\ 0\ 0\ 0\ 1\ 1]^T\in \Theta(N)$.
\end{itemize}

It is now straightforward to find theta-graphs in $\Theta(N)$ that are completed by the elements $[1\ 0\ 0\ 0\ 0\ 1]^T$, $[0\ 1\ 0\ 1\ 0\ 0]^T$, and $[0\ 0\ 1\ 0\ 1\ 0]^T$, so $\Theta(N)$ contains every vector of weight 1 or 2. Thus $\Theta(N)$ properly contains $M(K_7)$, so, by Proposition~\ref{mkorpgprop}, $\Theta(N)\cong P_6$.
\end{proof}

\section{The Main Result}
\label{main}

After a review of canonical tree decompositions, this section proves Theorem~\ref{mainresult}. For a set $\{M_1,M_2,\dots,M_n\}$ of matroids, a \emph{matroid-labelled tree} with vertex set $\{M_1,M_2,\dots,M_n\}$ is a tree $T$ such that
\begin{enumerate}
\item[(i)]{if $e$ is an edge of $T$ with endpoints $M_i$ and $M_j$, then $E(M_i)\cap E(M_j)=\{e\}$, and $\{e\}$ is not a separator of $M_i$ or $M_j$; and}
\item[(ii)]{$E(M_i)\cap E(M_j)$ is empty if $M_i$ and $M_j$ are non-adjacent.}
\end{enumerate}
The matroids $M_1, M_2,\dots,M_n$ are called the \emph{vertex labels} of $T$. Now suppose $e$ is an edge of $T$ with endpoints $M_i$ and $M_j$. We obtain a new matroid-labelled tree $T/e$ by contracting $e$ and relabelling the resulting vertex with $M_i\oplus_2 M_j$. As the matroid operation of 2-sum is associative, $T/X$ is well defined for all subsets $X$ of $E(T)$.

Let $T$ be a matroid-labelled tree for which $V(T)=\{M_1,M_2,\dots,M_n\}$ and\\
$E(T)=\{e_1,e_2,\dots,e_{n-1}\}$. Then $T$ is a \emph{tree decomposition} of a connected matroid $M$ if
\begin{enumerate}
\item[(i)]{$E(M)=(E(M_1)\cup E(M_2)\cup\cdots\cup E(M_n))-\{e_1,e_2,\dots,e_{n-1}\}$;}
\item[(ii)]{$\vert E(M_i)\vert\geq 3$ for all $i$ unless $\vert E(M)\vert<3$, in which case $n=1$ and $M=M_1$; and}
\item[(iii)]{$M$ labels the single vertex of $T/E(T)$.}
\end{enumerate}
In this case, the elements $\{e_1,e_2,\dots,e_{n-1}\}$ are the \emph{edge labels} of $T$. Cunningham and Edmonds~\cite{cunned} (see also~\cite[Theorem 8.3.10]{oxley}) proved the next theorem that says that $M$ has a \emph{canonical tree decomposition}, unique to within relabelling of the edges.

\begin{theorem}
\label{treedecomp}
Let $M$ be a $2$-connected matroid. Then $M$ has a tree decomposition $T$ in which every vertex label is $3$-connected, a circuit, or a cocircuit, and there are no two adjacent vertices that are both labelled by circuits or are both labelled by cocircuits. Moreover, $T$ is unique to within relabelling of its edges.
\end{theorem}

We now complete the proof of our main result.

\begin{proof}[Proof of Theorem~\ref{mainresult}]
Since circuits, cycle matroids of complete graphs, and projective geometries are in $\Theta_3$, by Proposition~\ref{parconnprop}, every matroid that can be built from such matroids by a sequence of parallel connections is in $\Theta_3$.

To prove the converse, we begin by noting that loops can be added via direct sums and that parallel elements can be added via parallel connections of circuits, so we may assume $M$ is simple. Let $T$ be the canonical tree decomposition of $M$. The proof is by induction on $\vert V(T)\vert$.

If $\vert V(T)\vert = 1$, then $M$ is 3-connected and the result holds by Theorem~\ref{mainres3conn}. Now assume $T$ has at least two vertices, and let $N$ be a matroid labelling a leaf of $T$. Since $M$ is simple, $N$ is not a cocircuit. We may now write $M=N\oplus_2 M_1$, where, by Corollary~\ref{2summands}, $N$ and $M_1$ are in $\Theta_3$. Thus, by Theorem~\ref{mainres3conn}, $N$ is a circuit, the cycle matroid of a complete graph of rank at least three, or a projective geometry of rank at least three. Moreover, by induction, $M_1$ is a parallel connection of circuits, cycle matroids of complete graphs, and projective geometries.

Let $N_1$ be the label of the neighbor of $N$ in $T$, and suppose $N_1$ is not a cocircuit. In this case, each of $N$ and $N_1$ is a circuit, the cycle matroid of a complete graph of rank at least three, or a projective geometry of rank at least three, and they are not both circuits. Therefore, if $p$ is the basepoint of the 2-sum $N\oplus_2 N_1$, there are circuits in $N$ and $N_1$ that form a theta-graph that is completed by $p$, a contradiction. Thus $N_1$ is a cocircuit.

Now let $k$ be the degree of $N_1$ in $T$. Evidently $N_1$ has at least $k$ elements, but, since $M$ is simple, $N_1$ has at most $k+1$ elements. If $k=2$, then $N_1$ has three elements as $N_1$ labels a vertex of $T$. Otherwise $k\geq 3$, so there are circuits in $M$ that form a theta-graph that is completed by an element of $N_1$. Hence $N_1$ has $k+1$ elements, and therefore corresponds to a parallel connection of its neighbors. It now follows that $M$ is the parallel connection of circuits, cycle matroids of complete graphs, and projective geometries.
\end{proof}


\section*{Acknowledgements}

The authors thank the referee for suggesting a number of improvements to the paper.

\end{document}